\documentclass[11pt]{article}

\usepackage{amsmath,amssymb,amsfonts,textcomp,amsthm,xifthen,psfrag,graphicx,color,MnSymbol}
\usepackage[T1]{fontenc}
\usepackage{hyperref}
\usepackage{tikz}
\usepackage{pgfplots}
\pgfplotsset{compat=1.18}

\newtheorem{theorem}{Theorem}
\newtheorem{lemma}[theorem]{Lemma}

\newtheorem{prop}[theorem]{Proposition}

\newtheorem{ass}[theorem]{Assumption}

\newcommand{\<}{\langle{}}
\renewcommand{\>}{\rangle}
\newcommand{\R}{\ensuremath{\mathbb{R}}}

\newcommand{\ip}[2]{\llangle#1\hspace*{.5mm},#2\rrangle}
\newcommand{\dual}[2]{\<#1\hspace*{.5mm},#2\>}
\newcommand{\vdual}[2]{(#1\hspace*{.5mm},#2)}

\newcommand{\dy}{\,\mathrm{d}y}
\newcommand{\Cstab}[1]{{C_\mathrm{stab}^{#1}}}

\newcommand{\tu}{{\hat u}}
\newcommand{\tw}{{\hat w}}
\newcommand{\ttheta}{{\hat\theta}}
\newcommand{\tn}{{\hat n}}
\newcommand{\tq}{{\hat q}}
\newcommand{\tm}{{\hat m}}

\newcommand{\bu}{\boldsymbol{u}}
\newcommand{\bw}{\boldsymbol{w}}
\newcommand{\bv}{\boldsymbol{v}}

\newcommand{\du}{{\delta\!u}}
\newcommand{\dw}{{\delta\!w}}
\newcommand{\dtheta}{{\delta\!\theta}}
\newcommand{\dn}{{\delta\!n}}
\newcommand{\dq}{{\delta\!q}}
\newcommand{\dm}{{\delta\!m}}

\newcommand{\dz}{{\delta\!z}}

\newcommand{\cT}{\mathcal{T}}
\newcommand{\ttt}{{\rm T}}

\title{A DPG method for the circular arch problem\thanks{Financially supported by ANID-Chile through Fondecyt project 1230013}}
\author{Norbert Heuer\thanks{Facultad de Matem\'{a}ticas, Pontificia 
Universidad Cat\'olica de Chile, Avenida Vicu\~na Mackenna 4860, Santiago, 
Chile, email: {\tt nheuer@uc.cl}}, Antti H.~Niemi\thanks{Civil Engineering Research Unit, University of Oulu, Pentti Kaiteran katu 1, Oulu, Finland,\\ email: {\tt antti.niemi@oulu.fi}}
}

\date{}

\begin{document}
\maketitle

\begin{abstract}
We consider an elastic model for a circular arch that incorporates membrane, transverse shear, and bending effects. The central line of the arch is partitioned into elements, and an ultra-weak variational formulation is developed alongside a discontinuous Petrov--Galerkin (DPG) approximation procedure based on so-called optimal test functions. The formulation uses discontinuous stress and displacement interpolations on the element mesh, with corresponding interface variables defined at the nodes. Theoretical analysis predicts optimal convergence rates for all quantities of interest, while also revealing potential error amplification influenced by the curvature of the arch and the imposed boundary conditions. The method is tested on examples with different support configurations. The numerical experiments confirm the theoretical predictions and further demonstrate that the accuracy of the DPG method can be improved by employing a suitably scaled test space norm.

\bigskip
\noindent
{\em Key words}: arch, discontinuous Petrov--Galerkin method, optimal test functions, robustness.

\noindent
{\em AMS Subject Classifications}: 74S05, 74K10, 65L10, 65L60 
\end{abstract}

\section{Introduction} \label{sec_intro}
In the domain of structural engineering and computational mechanics, the precise numerical analysis of thin structures such as circular arches has long been an area of both challenge and significance. This study embarks on an in-depth exploration of the Discontinuous Petrov--Galerkin (DPG) variational framework specifically tailored for the analysis of thin circular arches. The contribution fills a research void between the authors' previous investigation on DPG methodology for thin structures, which centered on beams, plates, and shells, see e.g.~\cite{niemi_discontinuous_2011,fuhrer_ultraweak_2018,fuhrer_dpg_2022}. 

In the longer course of events, the closest ancestor of the present contribution is probably the Petrov--Galerkin method for circular arches by Loula \emph{et al.}~which employs a mixed finite element approach based on the so-called Hellinger--Reissner principle \cite{LoulaFHM_87_SCA}. Even though Petrov--Galerkin methods have their origin in transport problems, they have certain benefits also in the numerical analysis of thin structures. As is well known by now, numerical methods for thin structures based on standard variational formulations of canonical energy principles are stable by construction, but they are susceptible to various locking effects which lead to sub-optimal convergence rates, or even to a total lack of convergence, see \cite{ashwell_limitations_1971} for an early discussion in the context of circular arches. Non-standard variational formulations are typically not so prone to numerical locking effects; however, ensuring discrete stability poses its own challenges, which can be addressed using Petrov--Galerkin formulations.

Other studies addressing numerical locking of arch structures have been carried out by Reddy and co-workers in \cite{reddy_convergence_1988,reddy_mixed_1992,arunakirinathar_mixed_1993} and by Chapelle in \cite{chapelle_locking-free_1997}. These studies take the axial membrane force and the transverse shear force as independent unknowns and employ mixed formulations in a somewhat similar manner as Arnold in his pioneering work \cite{arnold_discretization_1981} concerning straight beams. These mixed methods are closely related to the standard finite element method with reduced integration. They are also connected to the non-conforming approach that arises from geometrically approximating a curved arch by an assemblage of beam elements, as analyzed by Bernadou and Ducatel \cite{bernadou_approximation_1982} and Kikuchi and Tanizawa \cite{kikuchi_accuracy_1984}. Pitk\"{a}ranta has analyzed the method as a non-conforming finite element method in \cite{pitkaranta_mathematical_2003}.

Like in the earlier DPG formulation for straight beams \cite{niemi_discontinuous_2011}, the starting point of the present formulation is the first order system of equilibrium and constitutive equations of the arch. The stability analysis in the present paper is based on the adoption of the Babu\v{s}ka--Brezzi theory for ultra-weak variational formulations outlined in \cite{CarstensenDG_16_BSF}. The steps consist of bounding the trace operator associated to the interface variables from below and establishing well-posedness of the continuous adjoint problem of the original variational problem.  The former step is established by an explicit discrete characterization of the trace norm in same way as in \cite{FuehrerGH_21_LFD}. The latter step is then carried out by extending the analysis of Loula \emph{et~al.}~in \cite{LoulaFHM_87_SCA} to cover a range of boundary conditions.  Well-posedness and best approximation properties of the DPG scheme then follow.

The stability constant of the DPG method can be bounded uniformly with respect to the slenderness parameter, ensuring robustness in the thin arch limit. However, their dependence on the arch curvature parameter is more intricate and may lead to error amplification for deep arches under high-curvature conditions. Our numerical experiments show that this error growth can be effectively avoided by employing a suitably scaled test space norm. This approach is further supported by both experimental observations and theoretical stability analyses available in the literature, see e.g.~\cite{niemi_discontinuous_2011,zitelli_class_2011,demkowicz_wavenumber_2012,gopalakrishnan_dispersive_2014,demko_heuer_2013,chan_robust_2014,niemi_discontinuous_2013}.

The structure of the paper is as follows. In Section \ref{sec_DPG}, we introduce the model problem, notation, assumptions and summarize the theoretical stability results. Section \ref{sec_pf} is then devoted to the mathematical proofs of the results and the paper is concluded with numerical verification of the theoretical results in Section~\ref{sec_numerics}.

Throughout this paper, inequalities written with the symbols $\lesssim$ and $\gtrsim$ are understood to hold up to a generic positive constant that does not depend on the discretization or critical model parameters.


\section{Model problem, variational formulation, and DPG method} \label{sec_DPG}
An arch of length $L$, occupying the interval $[0,L]$ is considered. The constitutive relations for axial force $n$, shear force $q$, and bending moment $m$ are given by
\[
   n=EA(u'+\frac wR),\quad q=\kappa GA (w'-\frac uR -\theta),\quad m=EI\theta',
\]
where $E$ is the Young modulus, $G$ is the shear modulus, $\kappa$ is the shear correction coefficient, and $A$ and $I$ stand for the first and second moment of area of the arch cross section, respectively. The variables $u,w,\theta$ denote the axial, transverse, and rotational displacements, respectively. These are accompanied by equilibrium equations that reflect force and moment balance:
\[
   -n'-\frac qR = f_u,\quad -q' + \frac nR = f_w,\quad -m'-q=0.
\]
To facilitate analysis and computation, we follow the scaling approach introduced in \cite{LoulaFHM_87_SCA} and define the following dimensionless variables and data:
\begin{align*}
   \frac uL &\hookrightarrow u, & \frac wL &\hookrightarrow w, \\
   \frac {L^2}{EI}n &\hookrightarrow n, &
   \frac {L^2}{EI} q &\hookrightarrow q, & \frac L{EI}m &\hookrightarrow m,\\
   \frac {L^3}{EI} f_u &\hookrightarrow f_u, &
   \frac {L^3}{EI} f_w &\hookrightarrow f_w.
\end{align*}
Substituting these expressions and introducing the non-dimensional spatial variable $x=s/L\in [0,1]$ instead of the arc length $s\in [0,L]$, we arrive at the scaled system
\begin{subequations} \label{prob}
\begin{alignat}{4}
   \epsilon^2 n - u' - \lambda w &= 0,  \quad &\mu\epsilon^2 q - w' + \lambda u + \theta &= 0,
   &\quad m - \theta' &= 0,
   \label{p1}\\
   -n'-\lambda q  &= f_u,\quad &-q' + \lambda n &= f_w, &\quad -m'-q &= 0, \label{p2}
\end{alignat}
\end{subequations}
with the new parameters
\[
   \epsilon^2=\frac I{AL^2},\quad \lambda=\frac LR, \quad \mu=\frac E{\kappa G}.
\]
Here, $\epsilon$ is a small positive constant characterizing the slenderness of the arch, and the central angle $\lambda$ describes its curvature, or depth. Depending on the value of $\lambda$, the system describes different structural regimes: deep arch behavior for $\lambda \geq 1$, shallow arch behavior for $0.1 \leq \lambda \leq 1$, and bar-beam behavior when $\lambda \ll 1$, where curvature effects become negligible. As $\lambda \rightarrow 0$ (straight-beam limit), the present non-dimensional scaling degenerates, and a separate rescaling of the bar subproblem is required to recover a meaningful straight-bar formulation.

The parameter $\mu$ characterizes the relative stiffness of axial to shear deformation and can be interpreted as a material anisotropy parameter. In isotropic materials, it depends on Poisson's ratio and the shear correction factor. The limit $\mu \rightarrow 0$ ($\kappa \rightarrow \infty$) corresponds to infinite shear stiffness, and imposes the Euler--Bernoulli constraint, effectively eliminating transverse shear deformation from the model.
 
In what follows, we assume $\lambda\in (0,2\pi)$. We take $I=(0,1)$ as the non-dimensional domain and assume $f_u, f_w\in L_2(I)$.

Boundary conditions at the end points $x\in\{0,1\}$ are specified in terms of the displacement variables $u, w, \theta$ and the stress resultants $n, q, m$. The following standard cases are considered:
\begin{equation} \label{BC}
\begin{split}
   \text{fully clamped (``$c$''):}\quad & u(x)=w(x)=\theta(x)=0, \quad n,q,m\ \text{arbitrary at $x$},\\
   \text{pinned (``$p$''):}       \quad & u(x)=w(x)=m(x)=0,    \quad n,q,\theta\ \text{arbitrary at $x$},\\
   \text{sliding support (``$s$''):}      \quad & w(x)=\theta(x)=n(x)=0, \quad u,q,m\ \text{arbitrary at $x$},\\
   \text{axially fixed (``$d$''):}    \quad & u(x)=q(x)=m(x)=0, \quad w,\theta,n\ \text{arbitrary at $x$},\\
   \text{rotational restraint (``$r$''):}   \quad & \theta(x)=n(x)=q(x)=0, \quad u,w,m\ \text{arbitrary at $x$},\\
   \text{free (``$f$''):}         \quad & n(x)=q(x)=m(x)=0,    \quad u,w,\theta\ \text{arbitrary at $x$}.
\end{split}
\end{equation}
Problem \eqref{prob} has unknowns $(u,w,\theta,n,q,m)\in U^c:=H^1(I)^6$
(upper index $c$ refers to ``continuity''). It will be useful to separate the space
of kinematic quantities from the one of static quantities,
\[
    U^c = U^{1,c}\times U^{2,c} := H^1(I)^3 \times H^1(I)^3.
\]
The corresponding spaces with boundary conditions are denoted as
$U^c_{lr}=U^{1,c}_{lr}\times U^{2,c}_{lr}$ with $l,r\in\{c,p,s,d,r,f\}$, e.g.,
for $l=c$ (clamped at the left endpoint) and $r=f$ (free at the right endpoint)
\begin{equation} \label{BC_expl}
\begin{split}
   U^c_{cf} &= \{(u,w,\theta,n,q,m)\in U^c;\; u=w=\theta=0\ \text{at $0$},\ n=q=m=0\ \text{at $1$}\},\\
   U^{1,c}_{cf} &= \{(u,w,\theta)\in U^{1,c};\; u=w=\theta=0\ \text{at $0$}\},\\
   U^{2,c}_{cf} &= \{(n,q,m)\in U^{2,c};\; n=q=m=0\ \text{at $1$}\}.
\end{split}
\end{equation}
We decompose interval $I$ into $N$ sub-intervals (or elements) by selecting
nodes $0=x_0 < x_1 < x_2 < \ldots <x_N=1$, giving rise to a mesh
$\cT=\{I_j=(x_{j-1},x_j);\; j=1,\ldots,N\}$. The length of element $I_j$ is $h_j:=x_j-x_{j-1}$
($j=1,\ldots,N$). We test relations \eqref{prob} with test functions (also distinguishing between
kinematic and non-kinematic variables),
\[
   (\du,\dw,\dtheta,\dn,\dq,\dm)\in
   V := U^1\times U^2:=H^1(\cT)^3\times H^1(\cT)^3
\]
where
\[
   H^1(\cT) := \{z\in L_2(I);\; z|_{I_j}\in H^1(I_j),\ j=1,\ldots, N\}.
\]
(The equations in \eqref{p1} are tested with $\dn,\dq,\dm$ in this order, and the relations
in \eqref{p2} with $\du,\dw,\dtheta$.) Element-wise integration by parts yields
\begin{align}\label{VFa}
   &\vdual{u}{\dn'+\lambda\dq}_\cT + \vdual{w}{\dq'-\lambda\dn}_\cT + \vdual{\theta}{\dm'+\dq}_\cT
   \nonumber\\
   &+ \vdual{n}{\epsilon^2\dn+\du'+\lambda\dw}_\cT
   + \vdual{q}{\mu\epsilon^2\dq+\dw'-\lambda\du-\dtheta}_\cT
   + \vdual{m}{\dm+\dtheta'}_\cT
   \nonumber\\
   &-\dual{u}{\dn} -\dual{w}{\dq} -\dual{\theta}{\dm}
   -\dual{n}{\du} -\dual{q}{\dw}
   -\dual{m}{\dtheta}
   =
   \vdual{f_u}{\du} +\vdual{f_w}{\dw}.
\end{align}
Here, $\vdual{\cdot}{\cdot}_\cT$ is the element-wise $L_2$-bilinear form and
\[
   \dual{z}{\dz} := \sum_{j=1}^N \Bigl(z(x_{j})\dz|_{I_j}(x_{j}) - z(x_{j-1})\dz|_{I_j}(x_{j-1})\Bigr)
   \quad\text{for}\ z\in H^1(I),\ \dz\in H^1(\cT)
\]
where $\dz|_{I_j}(x_k)$ indicates that the trace of $\dz$ at $x_k$ is taken from within $I_j$.
Note that duality $\dual{z}{\dz}$ amounts to testing traces of $z$ at the nodes
with the jumps of $\dz$ there. In fact,
\[
   \dual{z}{\dz} = \dual{\gamma(z)}{\dz}\quad\forall \dz\in H^1(\cT)
\]
with traces $\gamma(z)$ identified as $(z(x_0),z(x_1),\ldots,z(x_N))^\top\in\R^{N+1}$.
We will use the corresponding notation
\[
   \dual{\hat z}{\dz} = \dual{\gamma(z)}{\dz} := \dual{z}{\dz}\quad\forall\dz\in H^1(\cT)
\]
for $\hat z=\gamma(z)\in\R^{N+1}$, and
$\gamma(\bu):=(\gamma(u),\gamma(w),\gamma(\theta),\gamma(n),\gamma(q),\gamma(m))$
for $\bu=(u,w,\theta,n,q,m)\in H^1(I)^6$.

We introduce the following spaces and norms (slightly abusing notation),
\begin{align*}
   &\qquad U := U_0\times \hat U,\quad
   U_0 := \oplus_{j=1}^6 L_2(I),\quad \hat U:= \R^{6N+6},\\
   &\|(u,w,\theta,n,q,m)\|^2 := \|u\|^2 + \|w\|^2 + \|\theta\|^2 + \|n\|^2 + \|q\|^2 + \|m\|^2,\\
   &\|(\hat u,\hat w,\hat\theta,\hat n,\hat q,\hat m)\|_\gamma^2
    := \|\hat u\|_\gamma^2 + \|\hat w\|_\gamma^2 + \|\hat\theta\|_\gamma^2
     + \|\hat n\|_\gamma^2 + \|\hat q\|_\gamma^2 + \|\hat m\|_\gamma^2,\\
   &\|\bu\|_U^2 := \|\bu_0\|^2 + \|\hat\bu\|_\gamma^2
   \quad\text{for}\ \bu=(\bu_0,\hat\bu),\
   \bu=(u,w,\theta,n,q,m)\in U_0,\ \hat\bu=(\hat u,\hat w,\hat\theta,\hat n,\hat q,\hat m)\in \hat U,\\
   &\|\bv\|_V^2 := \|\du\|_{1,\cT}^2 + \|\dw\|_{1,\cT}^2 + \|\dtheta\|_{1,\cT}^2
                 + \|\dn\|_{1,\cT}^2 + \|\dq\|_{1,\cT}^2 + \|\dm\|_{1,\cT}^2,\quad
   \|z\|_{1,\cT} := \sum_{j=1}^N \|z|_{I_j}\|_{1,I_j}^2\\
   &\quad\text{with}\quad
   \|z|_{I_j}\|_{1,I_j}^2 := \sum_{j=1}^N \|z|_{I_j}\|_{I_j}^2 + \|(z|_{I_j})'\|_{I_j}^2\quad
   \text{for}\ \bv=(\du,\dw,\dtheta,\dn,\dq,\dm)\in V,\ z\in H^1(\cT).
\end{align*}
Here, $\|\cdot\|_{I_j}$, $\|\cdot\|$ denote the $L_2(I_j)$ and $L_2(I)$-norms, respectively, and
\begin{equation} \label{trace_norm}
   \|\hat z\|_\gamma^2 :=
   \sum_{j=1}^N \frac {h_j}3 \bigl(z(x_{j-1})^2 + z(x_{j-1})z(x_j) + z(x_j)^2\bigr)
              + \frac 1{h_j} \bigl(z(x_{j-1}) - z(x_j)\bigr)^2
\end{equation}
for $\hat z=\gamma(z)$, $z\in H^1(I)$. Induced by \eqref{VFa}, and considering independent traces,
we consider the bilinear and linear forms
\begin{align*}
   b(\bu,\bv)
   &:=
   \vdual{u}{\dn'+\lambda\dq}_\cT + \vdual{w}{\dq'-\lambda\dn}_\cT + \vdual{\theta}{\dm'+\dq}_\cT
   \\
   &\quad+ \vdual{n}{\epsilon^2\dn+\du'+\lambda\dw}_\cT
   + \vdual{q}{\mu\epsilon^2\dq+\dw'-\lambda\du-\dtheta}_\cT
   + \vdual{m}{\dm+\dtheta'}_\cT
   \\
   &\quad-\dual{\tu}{\dn} -\dual{\tw}{\dq} -\dual{\ttheta}{\dm}
   -\dual{\tn}{\du} -\dual{\tq}{\dw} -\dual{\tm}{\dtheta},
   \\
   L(\bv) &:= \vdual{f_u}{\du} +\vdual{f_w}{\dw}
\end{align*}
for $\bu=(\bu_0,\hat\bu)\in U$, $\bu_0=(u,w,\theta,n,q,m)$, $\hat\bu=(\tu,\tw,\ttheta,\tn,\tq,\tm)$,
and $\bv=(\du,\dw,\dtheta,\dn,\dq,\dm)\in V$.
We obtain the following ultra-weak formulation of \eqref{prob} with boundary conditions
``$lr$'' (cf.~\eqref{BC}, \eqref{BC_expl}): \emph{Find $\bu\in U_{lr}$ such that}
\begin{equation} \label{VF}
   b(\bu,\bv)=L(\bv)\quad\forall \bv\in V
\end{equation}
with
\[
   U_{lr} := U_0\times \hat U_{lr},\quad
   \hat U_{lr} := \gamma(U^c_{lr}).
\]
Note that $\dim(\hat U_{lr})=6N$.

The DPG method is defined in the standard way. Employing trial-to-test operator
\[
   \ttt:\;U\to V:\quad
   \ip{\ttt\bu}{\bv}_V = b(\bu,\bv)\quad\forall\bv\in V
\]
with canonical inner product $\ip{\cdot}{\cdot}_V$ in $V=H^1(\cT)^6$,
and selecting a finite-dimensional subspace $U_h\subset U_{lr}$, a DPG approximation
$\bu_h$ is given by solving
\begin{equation} \label{DPG}
   \bu_h\in U_h:\quad b(\bu_h,\bv) = L(\bv)\quad\forall\bv\in\ttt(U_{lr}).
\end{equation}
We will consider particular combinations of boundary conditions and corresponding choices
of parameters $\epsilon,\mu,\lambda$ to prove  the well-posedness of variational formulation
\eqref{VF} and convergence of scheme \eqref{DPG}. We formulate these choices as an assumption.

\begin{ass} \label{ass}
We assume that the boundary conditions and parameters $\epsilon>0$, $\mu\ge 0$, $\lambda\in (0,2\pi)$
are selected in such a way that the parameters are bounded and that the  number $\Cstab{}$ defined by
\begin{align*}
   &\Cstab{}=\max\left\{\min\{\epsilon^{-2},C_n(\lambda)\},
                      \min\{\mu^{-1}\epsilon^{-2},C_q(\lambda)C_n(\lambda)\},1 \right\}
   &&(lr\in\{cc,cp,pc,cs,sc,ps,sp\}),\\
   &\Cstab{}=\max \{\min\{\epsilon^{-2},C_q^{(0)}(\lambda)C_n(\lambda)\},1\} |\cos \lambda|^{-2}
   && (lr\in\{sd,ds\}),\\
   &\Cstab{}=\max\{\min\{\epsilon^{-2},C_q^{(0)}(\lambda)C_n(\lambda)\},1\}   && (lr\in\{cd,dc\}),\\
   &\Cstab{}=1   && (lr\in\{cr,rc,cf,fc,pr,rp\})
\end{align*}
is finite. Here,
\[
\begin{aligned}
	C_n(\lambda) := \lambda^2+ \bigl(\frac 1\lambda + \lambda\bigr)^2, \quad
	C_q(\lambda) := \frac {\lambda+|\sin \lambda|}{\lambda-|\sin\lambda|}, \quad
	C_q^{(0)}(\lambda) := \frac{2\lambda - \sin (2\lambda)}{2\lambda + \sin (2\lambda)}.
\end{aligned}
\]
\end{ass}
The functions $C_n(\lambda)$, $C_q(\lambda)$ and $C_q^{(0)}(\lambda)$ encode curvature-dependent stability effects and appear in the proof of Lemma~\ref{la_coercive}, see Figure~\ref{fig:stability}.
\begin{figure}[ht!]
	\centering
	\begin{tikzpicture}
		\begin{axis}[
			xmin=0,
			xmax=6,
			ymin=0,
			ymax=80,
			width=0.9\textwidth,
			height=0.55\textwidth,
			xlabel={$\lambda$},
			domain=0.1:2*pi,
			unbounded coords=jump,
			samples=900,
			grid=both,
			xtick={0,1,2,3,4,5,6},
			legend style={
				draw=none,
				fill=none,
				at={(0.5,-0.18)},
				anchor=north,
				legend columns=3,
				/tikz/every even column/.append style={column sep=8pt},
			},
			]
			
			\addplot[thick] {x^2 + (1/x + x)^2};
			\addlegendentry{$C_n(\lambda)=\lambda^2+\left(\frac{1}{\lambda}+\lambda\right)^2$}
			
			\addplot[thick, dashed]
			{(x + abs(sin(deg(x))))/(x - abs(sin(deg(x))))};
			\addlegendentry{$C_q(\lambda)=\dfrac{\lambda+|\sin\lambda|}{\lambda-|\sin\lambda|}$}
			
			\addplot[thick, dotted]
			{(2*x - sin(deg(2*x)))/(2*x + sin(deg(2*x)))};
			\addlegendentry{$C_q^0(\lambda)=\dfrac{2\lambda-\sin(2\lambda)}{2\lambda+\sin(2\lambda)}$}
			
		\end{axis}
	\end{tikzpicture}
	\caption{Stability constants as functions of the curvature parameter $\lambda$.}
	\label{fig:stability}
\end{figure}
\begin{theorem} \label{thm}
Let $f_u,f_w\in L_2(I)$ be given and assume that Assumption~\ref{ass} holds true.
Problem \eqref{VF} is well posed with solution $\bu\in U_{lr}$ bounded as
\begin{equation} \label{stab}
   \|\bu\|_U \lesssim \Cstab{}(\|f_u\| + \|f_w\|).
\end{equation}
Furthermore, DPG scheme \eqref{DPG} is well-posed and converges quasi-optimally,
\[
   \|\bu-\bu_h\|_U \lesssim \Cstab{} \inf\{\|\bu-\bw\|_U;\; \bw\in U_h\}.
\]
In both estimates, the hidden constant is independent of
$f_u,f_w,\epsilon,\mu,\lambda$, and mesh $\cT$, and $\Cstab{}$ is the number from
Assumption~\ref{ass}.
\end{theorem}

A proof is given in the next section.

\section{Proofs} \label{sec_pf}

We start with characterizing the trace norm of $H^1(I)$. Afterwards we study the adjoint problem
to \eqref{prob}. We finish with a proof of Theorem~\ref{thm}.

\begin{lemma} \label{la_trace}
The norms
\[
   \|\hat z\|_{-1,\cT}:=
   \sup_{0\not=\dz\in H^1(\cT)}
   \frac{\dual{\hat z}{\dz}}{\|\dz\|_{1,\cT}} \quad \text{and} \quad \|\hat z\|_\gamma
\]
are uniformly equivalent in $\R^{N+1}$ with respect to $N$ and $\cT$.
\end{lemma}

\begin{proof}
We follow similar steps as in the proof of \cite[Lemma~3]{FuehrerGH_21_LFD} where a higher order
case was considered.
Let $\hat z\in \R^{N+1}$ be given. By standard estimates,
cf.~\cite[Theorem~2.3]{CarstensenDG_16_BSF}, \cite[Lemma~3]{FuehrerGH_21_LFD}, one finds that
\[
   \|\hat z\|_{-1,\cT} = \inf\{\|z\|_{1,\cT};\; z\in H^1(I),\ \gamma(z)=\hat z\}.
\]
The infimum norm on the right-hand side is localizable with respect to the elements.
We therefore characterize, for $(z_0,z_1)\in\R^2$ and $I_h:=(0,h)$ with $h>0$,
the trace norm of $H^1(I_h)$ as
\[
   \|(z_0,z_1)\|_h
   := \inf\left\{\|z\|_{1,I_h}:=\bigl(\|z\|_{I_h}^2+\|z'\|_{I_h}^2\bigr)^{1/2};\;
            z\in H^1(I_h),\ z(0)=z_0,\ z(h)=z_1\right\}.
\]
The latter infimum is obtained by solving $-z''+z=0$ in $I_h$ subject to $z(0)=z_0$ and $z(h)=z_1$.
As in \cite[Lemma~3]{FuehrerGH_21_LFD} one sees that the $H^1(I_h)$-norm of this function
is uniformly equivalent to the norm of $\tilde z$ where
$\tilde z''=0$ and $\tilde z(0)=z_0$, $\tilde z(h)=z_1$: $\tilde z(x)=z_0+(z_1-z_0)x/h$. Calculating
\[
   \|\tilde z\|_{1,I_h}^2 = \frac h3\begin{pmatrix}z_0\\z_1\end{pmatrix}^\top
                   \begin{pmatrix}1 & 1/2\\ 1/2 & 1\end{pmatrix}\begin{pmatrix}z_0\\z_1\end{pmatrix}
                 + \frac 1h\begin{pmatrix}z_0\\z_1\end{pmatrix}^\top
                   \begin{pmatrix}1 & -1\\ -1 & 1\end{pmatrix}\begin{pmatrix}z_0\\z_1\end{pmatrix},
\]
and recalling definition \eqref{trace_norm} of $\|\cdot\|_\gamma$, the statement follows.
\end{proof}

\subsection{Analysis of the adjoint problem} \label{sec_adj}

Bilinear form $b(\cdot,\cdot)$ reveals the continuous adjoint problem of \eqref{prob}:
\emph{Given $g_u,g_w,g_\theta,g_n,g_q,g_m\in L_2(I)$, find $(u,w,\theta,n,q,m)\in U^c_{lr}$ with}
\begin{subequations} \label{adj}
\begin{alignat}{4}
   \epsilon^2 n+u'+\lambda w &= g_n,  \quad &\mu\epsilon^2 q+w'-\lambda u -\theta &= g_q,
   &\quad m+\theta' &= g_m,
   \label{a1}\\
   n'+\lambda q  &= g_u,\quad &q' - \lambda n &= g_w, &\quad m'+q &= g_\theta. \label{a2}
\end{alignat}
\end{subequations}
We analyze problem \eqref{adj} in a mixed setting, as studied by Loula \emph{et al.}~\cite{LoulaFHM_87_SCA}.

We recall notation $U^c_{lr}=U^{1,c}_{lr}\times U^{2,c}_{lr}$ for the spaces of
kinematic and non-kinematic variables. Introducing the bilinear forms
\begin{align*}
   a(n,q,m;\dn,\dq,\dm) &:=
   \epsilon^2 \vdual{n}{\dn} + \mu\epsilon^2 \vdual{q}{\dq} + \vdual{m}{\dm},\\
   c(n,q,m;\du,\dw,\dtheta) &:=
   \vdual{n}{\du'+\lambda \dw} + \vdual{q}{\dw'-\lambda \du-\dtheta} + \vdual{m}{\dtheta'},
\end{align*}
problem~\eqref{adj} has the mixed formulation:
\emph{Find $(n,q,m)\in L_2(I)^3$ and $(u,w,\theta)\in U^{1,c}_{lr}$ such that}
\begin{subequations} \label{adj_mixed}
\begin{alignat}{2}
   &a(n,q,m;\dn,\dq,\dm) + c(\dn,\dq,\dm;u,w,\theta)
   &&= \vdual{g_n}{\dn} + \vdual{g_q}{\dq} + \vdual{g_m}{\dm},\\
   &c(n,q,m;\du,\dw,\dtheta) &&= -\vdual{g_u}{\du} - \vdual{g_w}{\dw} - \vdual{g_\theta}{\dtheta}
\end{alignat}
\end{subequations}
\emph{for any $(\dn,\dq,\dm)\in L_2(I)^3$ and $(\du,\dw,\dtheta)\in U^{1,c}_{lr}$.}

We prove the well-posedness of \eqref{adj_mixed} in Proposition~\ref{prop_adj} below.
To this end we introduce the kernel of bilinear form $c(\cdot,\cdot)$,
\begin{equation} \label{K}
   K_{lr} :=
   \{(n,q,m)\in L_2(I)^3;\; c(n,q,m;\du,\dw,\dtheta)=0
   \ \forall (\du,\dw,\dtheta)\in U^{1,c}_{lr}\}.
\end{equation}

\begin{lemma} \label{la_coercive}
Bilinear form $a(\cdot,\cdot)$ is coercive on $K_{lr}$:
\[
   a(n,q,m;n,q,m) \gtrsim
   \alpha_n\|n\|^2 + \alpha_q\|q\|^2 + \|m\|^2
   \quad\forall (n,q,m)\in K_{lr}
\]
holds true with a hidden constant that is independent of $\epsilon, \mu,\lambda$, where
\begin{align*}
   &\alpha_n=\max\{\epsilon^2,C_n(\lambda)^{-1}\},\ \alpha_q=\max\{\mu\epsilon^2,C_q(\lambda)^{-1}C_n(\lambda)^{-1}\}
                                                          &&\text{for any boundary condition},\\
   &\alpha_n=\max\{\epsilon^2,C_n(\lambda)^{-1}\},\ \alpha_q= C_q^{(0)}(\lambda)^{-1}C_n(\lambda)^{-1}  && \text{if}\ q(0)q(1)=0,\\
   &\alpha_n=\alpha_q=1 && \text{if}\ n(0)=q(0)=0\ \text{or}\ n(1)=q(1)=0.
\end{align*}
Here $C_n(\lambda)$, $C_q(\lambda)$ and $C_q^{(0)}(\lambda)$ are the curvature-dependent stability constants introduced in Assumption \ref{ass} and worked out explicitly in the proof below.
\end{lemma}

\begin{proof}
For any boundary condition, $K_{lr}$ is a subspace (possibly adding boundary conditions) of
\[
   K_{cc} = 
   \{(n,q,m)\in L_2(I)^3;\; n'+\lambda q=0,\ q'-\lambda n=0,\ m'+q=0\}.
\]
Here, derivatives are taken in the distributional sense, and imply that $K_{cc}\subset H^1(I)^3$.

Now let $(n,q,m)\in K_{cc}$ be given. Combining relations $n'+\lambda q=0$ and $q'-\lambda n=0$ we obtain
$q''=-\lambda^2 q$ so that $q(x)=A \cos(\lambda x)+B \sin(\lambda x)$ with $A,B\in\R$.
Then, $n(x)=q'(x)/\lambda=-A\sin(\lambda x)+B\cos(\lambda x)$. Furthermore,
$n(x)=n(0)-\lambda\int_0^x q=n(0)-B-A\sin(\lambda x)+B\cos(\lambda x)$, and $B=n(0)$, $A=q(0)$, i.e.,
\begin{equation} \label{rel_nq}
   n(x) = n_0 \cos(\lambda x)-q_0\sin(\lambda x),\quad
   q(x) = q_0 \cos(\lambda x)+n_0\sin(\lambda x)
\end{equation}
with $n_0:=n(0)$ and $q_0:=q(0)$. We calculate the $L_2(I)$-norms of $n$ and $q$,
\begin{subequations} \label{norms_nq}
\begin{align}
   \|n\|^2 &= \frac{n_0^2}2 \Bigl(1+\frac{\sin(2\lambda)}{2\lambda}\Bigr)
            + \frac{q_0^2}2 \Bigl(1-\frac{\sin(2\lambda)}{2\lambda}\Bigr)
            - n_0q_0 \frac{\sin^2\lambda}\lambda
            = \frac 12 (n_0, q_0) \Lambda_n (n_0, q_0)^\top,\\
   \|q\|^2 &= \frac{n_0^2}2 \Bigl(1-\frac{\sin(2\lambda)}{2\lambda}\Bigr)
            + \frac{q_0^2}2 \Bigl(1+\frac{\sin(2\lambda)}{2\lambda}\Bigr)
            + n_0q_0 \frac{\sin^2\lambda}\lambda
            = \frac 12 (n_0, q_0) \Lambda_q (n_0, q_0)^\top
\end{align}
\end{subequations}
where
\[
   \Lambda_n = \begin{pmatrix} 1+\frac{\sin(2\lambda)}{2\lambda} & -\frac{\sin^2\lambda}\lambda\\
                                -\frac{\sin^2\lambda}\lambda & 1-\frac{\sin(2\lambda)}{2\lambda}
                \end{pmatrix},\qquad
   \Lambda_q = \begin{pmatrix} 1-\frac{\sin(2\lambda)}{2\lambda} & \frac{\sin^2\lambda}\lambda\\
                                \frac{\sin^2\lambda}\lambda & 1+\frac{\sin(2\lambda)}{2\lambda}
                \end{pmatrix}
             = \begin{pmatrix} 0 & -1\\ 1 & 0\end{pmatrix} \Lambda_n
               \begin{pmatrix} 0 & 1\\ -1 & 0\end{pmatrix}.
\]
The eigenvalues of both matrices are
\[
     1\pm\frac 1{2\lambda}\sqrt{\sin^2(2\lambda)+4\sin^4\lambda}
   = 1\pm\frac {\sin\lambda}{\lambda}>0.
\]
It follows that
\begin{align} \label{est_qn}
   \|q\|^2 \le C_q(\lambda) \|n\|^2
\end{align}
with
\[
C_q(\lambda) = \frac {\lambda+|\sin\lambda|}{\lambda-|\sin\lambda|}.
\]
Next we bound $\|n\|$ by $\|m\|$. We use relations $m'=-q=n'/\lambda$ to calculate
$m(x) = m(0) + (n(x)-n(0))/\lambda$, that is, $n(x)=n(0)+\lambda(m(x)-m(0))$.
Testing the last relation with $m$ and $n$, and denoting $m_0:=m(0)$,
$\bar n:=\vdual{n}{1}$, $\bar m:=\vdual{m}{1}$, we find that
\begin{align} \label{rel_nm}
   \vdual{n}{m} &= n_0\bar m+\lambda\|m\|^2 -\lambda m_0\bar m,\nonumber\\
   \|n\|^2 &= n_0\bar n+\lambda \vdual{m}{n}-\lambda m_0\bar n
            = n_0\bar n+\lambda n_0\bar m+\lambda^2\|m\|^2 -\lambda^2 m_0\bar m-\lambda m_0\bar n,
\end{align}
that is,
\[
   \|n\|^2 = \lambda^2 \|m\|^2
           + \frac 12 (\bar n, n_0, \bar m, m_0) \Lambda (\bar n, n_0, \bar m, m_0)^\top
\]
with
\[
   \Lambda := \begin{pmatrix} 0 & 1 & 0 & -\lambda\\ 1 & 0 & \lambda & 0\\
                              0 & \lambda & 0 & -\lambda^2\\ -\lambda & 0 & -\lambda^2 & 0
              \end{pmatrix}.
\]
This matrix has the eigenvalues $0$ with multiplicity $2$, and $\pm(1+\lambda^2)$.
The eigenspace of the only positive eigenvalue is generated by $(1,1,\lambda,-\lambda)^\top$.
Therefore, bounding $|\bar m|=|\vdual{m}{1}|\le \|m\|$,
\begin{align} \label{est_nm}
   \|n\|^2
   \le
   \lambda^2\|m\|^2 + \frac {1+\lambda^2}2 \bigl( \frac 2{\lambda^2} + 2\bigr) \bar m^2
   \le C_n(\lambda) \|m\|^2,
\end{align}
where the first upper bound is obtained for
$(\bar n,n_0,\bar m,m_0)^\top=\bar m (\lambda^{-1},\lambda^{-1},1,-1)^\top$
and
\[
C_n(\lambda) = \lambda^2+ \bigl(\frac 1\lambda + \lambda\bigr)^2.
\]
Next, since 
$a(n,q,m;n,q,m)=\epsilon^2\|n\|^2+\mu\epsilon^2\|q\|^2+\|m\|^2$ for any $(n,q,m)\in L_2(I)^3$,
the first general statement follows. In the following we consider specific boundary conditions.

{\bf Case $q(0)q(1)=0$.}
If $q(0)=0$, relations \eqref{norms_nq} give $\|q\|^2 = C_q^{(0)}(\lambda) \|n\|^2$. Together with bound \eqref{est_nm} we find that $\|q\|^2\lesssim C_q^{(0)}(\lambda)C_n(\lambda) \|m\|^2$.
Furthermore, $\epsilon^2\|n\|^2$ is controlled by the bilinear form.
For $q(1)=0$ instead of $q(0)=0$ the result follows analogously.

{\bf Case $n(0)=q(0)=0$ or $n(1)=q(1)=0$.}
If $n(0)=q(0)=0$ then representation \eqref{rel_nq} shows that $n=q=0$.
It follows the trivial bound $a(n,q,m;n,q,m)\ge \|n\|^2+\|q\|^2+\|m\|^2$.
The case $n(1)=q(1)=0$ holds by symmetry.

\end{proof}

\begin{lemma} \label{la_infsup}
Bilinear form $c(\cdot,\cdot)$ satisfies the inf-sup condition
\begin{equation} \label{c_infsup}
   \sup_{0\not=(\dn,\dq,\dm)\in L_2(I)^3}
   \frac{c(\dn,\dq,\dm;u,w,\theta)}{\bigl(\|\dn\|^2+\|\dq\|^2+\|\dm\|^2\bigr)^{1/2}}
   \ge \beta \bigl(\|u\|_1^2+\|w\|_1^2+\|\theta\|_1^2\bigr)^{1/2}
   \quad\forall (u,w,\theta)\in H^1(I)^3
\end{equation}
uniformly for $\lambda\in(0,2\pi)$ if $\theta(0)\theta(1)=0$ with
\begin{align*}
   &\beta(\lambda)=O(1)              &&\text{if}\ u(0)=w(0)=0\ \text{or}\ u(1)=w(1)=0,\\
   &\beta(\lambda)=O\left(|\cos \lambda|\right)
                          && \text{if}\ u(0)=w(1)=0\ \text{or}\ u(1)=w(0)=0.
\end{align*}
\end{lemma}

\begin{proof}
We slightly generalize the proof of \cite[Lemma 3.3]{LoulaFHM_87_SCA} where the case
``$lr=cf$'' is studied (and obviously applies to ``$lr=fc$'').
Given $(u,w,\theta)\in H^1(I)^3$ we denote
\(
   g_1:= u'+\lambda w,\quad g_2:= w'-\lambda u - \theta
\)
so that $u$ and $w$ satisfy the ODE system
\[
   u'+\lambda w=g_1,\quad w'-\lambda u= \tilde g_2:= g_2+\theta
\]
with solution
\begin{equation} \label{sol_uw}
\begin{split}
   u(x) &= G_1^{\cos}(x) - G_2^{\sin}(x) - c_1\sin(\lambda x) + c_2\cos(\lambda x),\\
   w(x) &= G_1^{\sin}(x) + G_2^{\cos}(x) + c_1\cos(\lambda x) + c_2\sin(\lambda x).
\end{split}
\end{equation}
Here,
\begin{align*}
   G_1^{\cos}(x) &:= \int_0^x g_1(y)\cos\bigl(\lambda(x-y)\bigr)\dy,\quad
   G_2^{\sin}(x) := \int_0^x\tilde g_2(y)\sin\bigl(\lambda(x-y)\bigr)\dy,\\
   G_1^{\sin}(x) &:= \int_0^x g_1(y)\sin\bigl(\lambda(x-y)\bigr)\dy,\quad
   G_2^{\cos}(x) := \int_0^x \tilde g_2(y)\cos\bigl(\lambda(x-y)\bigr)\dy
\end{align*}
and $c_1,c_2\in\R$ are appropriate constants. Of course, $c_1=w(0)$, $c_2=u(0)$.

{\bf Case: $u$ and $w$ vanish at the same endpoint.}
If $u(0)=w(0)=0$, the case presented in \cite{LoulaFHM_87_SCA}, then $c_1=c_2=0$.
In the case that $u(1)=w(1)=0$ the analogous proof applies by symmetry.
But following the current path we calculate, by inverting matrix 
$\begin{pmatrix}-\sin\lambda & \cos\lambda \\ \cos\lambda & \sin\lambda\end{pmatrix}$,
\begin{align*}
   c_1 &= \sin\lambda\, G_1^{\cos}(1)-\sin\lambda\, G_2^{\sin}(1)
        - \cos\lambda\, G_1^{\sin}(1)-\cos\lambda\, G_2^{\cos}(1),
   \\
   c_2 &= -\cos\lambda\, G_1^{\cos}(1)+\cos\lambda\, G_2^{\sin}(1)
        - \sin\lambda\, G_1^{\sin}(1)-\sin\lambda\, G_2^{\cos}(1).
\end{align*}
In both cases we have with $|\sin \lambda|+|\cos \lambda| \le \sqrt{2}$ that
\[
   |c_1|+|c_2| \le \sqrt{2}\bigl(|G_1^{\cos}(1)| + |G_2^{\sin}(1)| + |G_1^{\sin}(1)| + |G_2^{\cos}(1)|\bigr)
               \lesssim \|g_1\| + \|g_2\| + \|\theta\|
\]
uniformly for $\lambda\in (0,2\pi)$. Representation \eqref{sol_uw} then shows that
\begin{align*}
   \|u\|, \|w\| &\le \|g_1\| + \|g_2\| + \|\theta\| + |c_1| + |c_2|
                 \lesssim \|g_1\| + \|g_2\| + \|\theta\|,\\
   \|u'\| &\le (1+\lambda)\|g_1\| + \lambda\|g_2\| + \lambda \|\theta\| + \lambda|c_1| + \lambda|c_2|
           \lesssim (1+\lambda)\|g_1\| + \lambda\|g_2\| + \lambda \|\theta\|, \\
   \|w'\| &\le \lambda\|g_1\| + (1+\lambda)\|g_2\| + (1+\lambda) \|\theta\| + \lambda|c_1| + \lambda|c_2|
           \lesssim \lambda\|g_1\| + (1+\lambda)\|g_2\| + (1+\lambda) \|\theta\|.
\end{align*}
Finally, $\|\theta\|\le \|\theta'\|$ for $\theta(0)\theta(1)=0$ by the Poincar\'e inequality.
We conclude that
\begin{align*}
   &\sup_{0\not=(\dn,\dq,\dm)\in L_2(I)^3}
   \frac{c(\dn,\dq,\dm;u,w,\theta)}{\bigl(\|\dn\|^2+\|\dq\|^2+\|\dm\|^2\bigr)^{1/2}}\\
   &= \bigl( \|u'+\lambda w\|^2 + \|w'-\lambda u-\theta\| + \|\theta'\|^2\bigr)^{1/2}
   \ge \beta \bigl(\|u\|_1^2 + \|w\|_1^2 + \|\theta\|_1^2\bigr)^{1/2}
\end{align*}
holds with a positive constant $\beta$ uniformly for $\lambda\in (0,2\pi)$.

{\bf Case: $u$ and $w$ vanish at different endpoints.}
If $u(0)=w(1)=0$, relations \eqref{sol_uw} show that
\[
   c_2=0,\quad c_1\cos\lambda = -G_1^{\sin}(1)- G_2^{\cos}(1).
\]
Analogously, if $u(1)=w(0)=0$ then
\[
   c_1=0,\quad c_2\cos\lambda = -G_1^{\cos}(1) + G_2^{\sin}(1).
\]
In both cases, bounding one of the constants requires division by $\cos\lambda$ and shows
\[
   \bigl(\|u\|_1^2+\|w\|_1^2+\|\theta\|_1^2\bigr)^{1/2}
   \lesssim | \cos \lambda|^{-1} \bigl(\|g_1\|^2 + \|g_2\|^2 + \|\theta'\|^2\bigr)^{1/2}
\]
that we obtain analogously to the first case.
This yields \eqref{c_infsup} with $\beta=O\left(|\cos \lambda|\right)$, and finishes the proof.
\end{proof}

\begin{prop} \label{prop_adj}
Let $g_u,g_w,g_\theta,g_n,g_q,g_m\in L_2(I)$ be given and assume that Assumption~\ref{ass} holds true.
Problem \eqref{adj} has a unique weak solution $(u,w,\theta,n,q,m)\in U^c_{lr}$. It is bounded as
\[
   \|u\|_1 + \|w\|_1 + \|\theta\|_1 + \|n\|_1 + \|q\|_1 + \|m\|_1
   \lesssim
   \Cstab{} (\|g_u\| + \|g_w\| + \|g_\theta\| + \|g_n\| + \|g_q\| + \|g_m\|)
\]
with a hidden constant that is independent of all the parameters and data
$g_u,g_w,g_\theta,g_n,g_q,g_m$, and where $\Cstab{}$ is the number defined in Assumption~\ref{ass}.
\end{prop}

\begin{proof}
We consider the saddle-point formulation \eqref{adj_mixed} of problem \eqref{adj}. Since $\epsilon$ and $\mu$ are bounded, bilinear forms $a(\cdot,\cdot):\;L_2(I)^3\times L_2(I)^3\to\R$ and
$c(\cdot,\cdot):\;L_2(I)^3\times H^1(I)^3\to\R$ are uniformly bounded.
By Lemma~\ref{la_coercive}, bilinear form $a(\cdot,\cdot)$ is coercive on the kernel $K_{lr}$ of
$c(\cdot,\cdot)$ with number $\alpha:=c\min\{\alpha_n,\alpha_q,1\}$ for a fixed constant $c>0$.
Furthermore, bilinear form $c(\cdot,\cdot)$ satisfies inf-sup property \eqref{c_infsup}
with number $\beta>0$ according to the details given by Lemma~\ref{la_infsup}.
By the Babu\v{s}ka--Brezzi theory, cf.~\cite[Theorem 4.2.3]{BoffiBF_13_MFE},
system \eqref{adj_mixed} has a unique solution which is bounded as
\begin{align*}
   &\|n\| + \|q\| + \|m\| \lesssim
   \alpha^{-1}(\|g_n\| + \|g_q\| + \|g_m\|) + \alpha^{-1/2}\beta^{-1}(\|g_u\| + \|g_w\| + \|g_\theta\|),\\
   &\|u\|_1 + \|w\|_1 + \|\theta\|_1 \lesssim
   \alpha^{-1/2}\beta^{-1}(\|g_n\| + \|g_q\| + \|g_m\|) + \beta^{-2}(\|g_u\| + \|g_w\| + \|g_\theta\|).
\end{align*}
Of course, the solution of \eqref{adj_mixed} solves \eqref{adj} in the weak sense. Therefore, the
regularities $n,q,m\in H^1(I)$ and bounds for the $H^1(I)$-norms of $n,q,m$ are
obtained by relations \eqref{a2},
\[
   \|n'\|\le \lambda\|q\|+ \|g_u\|,\quad
   \|q'\|\le \lambda\|n\| + \|g_w\|,\quad
   \|m'\|\le \|q\|+\|g_\theta\|.
\]
Combination of all the estimates gives
\begin{align*}
   &\|u\|_1 + \|w\|_1 + \|\theta\|_1 + \|n\|_1 + \|q\|_1 + \|m\|_1\\
   &\lesssim
   \max\{\alpha^{-1},\alpha^{-1/2}\beta^{-1},\beta^{-2}\}
   (\|g_u\| + \|g_w\| + \|g_\theta\| + \|g_n\| + \|g_q\| + \|g_m\|).
\end{align*}
Filling in the information for $\alpha$ from Lemma~\ref{la_coercive} and for $\beta$
from Lemma~\ref{la_infsup}, we obtain an upper bound with stability constant $\Cstab{}$
as defined in Assumption~\ref{ass}.
The boundary conditions implying $(n,q,m)\in U^{2,c}_{cf}$ are immediate.
\end{proof}

\subsection{Proof of Theorem~\ref{thm}} \label{sec_pf_thm}

The well-posedness of \eqref{VF} follows by the Babu\v{s}ka--Brezzi theory.
Let us recall the ingredients.
Using Lemma~\ref{la_trace} we see that bilinear form $b(\cdot,\cdot)$ and functional $L$
are uniformly bounded in the corresponding spaces and norms. The injectivity
\[
   \bv\in V:\ b(\bu,\bv)=0\quad\forall \bu\in U\quad\Rightarrow\quad\bv=0
\]
can be seen as follows. First, selecting $\bu=(\bu_0,\hat\bu)$ with $\bu_0=0$ and
arbitrary $\hat\bu\in\R^{6N+6}$ (subject to the imposed boundary conditions) it is clear that
$\bv\in U^c_{lr}$. Then, $\bv$ solves adjoint problem \eqref{adj} and its mixed form
\eqref{adj_mixed} with homogeneous data. By Proposition~\ref{prop_adj}, $\bv=0$. Now, inf-sup property
\begin{equation} \label{infsup}
   \sup_{0\not=\bv\in V} \frac {b(\bu,\bv)}{\|\bv\|_V}
   \gtrsim
   \|\bu\|_U
   \quad\forall \bu\in U^c_{lr}
\end{equation}
is implied by
\begin{equation} \label{infsup1}
   \sup_{0\not=\bv\in U^c_{lr}}
   \frac{b((\bu_0,0),\bv)}{\|\bv\|_V}
   \gtrsim \|\bu_0\| \quad\forall \bu_0\in U_0
\end{equation}
and
\begin{equation} \label{infsup2}
   \sup_{0\not=\bv\in V} \frac {b((0,\hat\bu),\bv)}{\|\bv\|_V}
   \gtrsim
   \|\hat\bu\|_\gamma\quad\forall \hat\bu\in \R^{6N+6},
\end{equation}
see \cite[Theorem 3.3]{CarstensenDG_16_BSF}. Here, we use that
\[
   \{\bv\in V;\; b((0,\hat\bu),\bv)=0\ \forall\hat\bu\} = U^c_{lr},
\]
already observed before.
Relation \eqref{infsup2} holds by Lemma~\ref{la_trace} with a positive constant that is independent
of all the involved parameters, and \eqref{infsup1} follows from Proposition~\ref{prop_adj}
with intrinsic number $\Cstab{-1}$ (multiplied with a fixed positive constant) where $\Cstab{}>0$
is as specified in Assumption~\ref{ass}.
This proves the well-posedness of \eqref{VF} and boundedness \eqref{stab} of its solution.

From~\cite{DemkowiczG_11_ADM} we know that the DPG scheme minimizes the residual
in the dual test norm $\|\cdot\|_{V'}$. The quasi-optimal error estimate in the $U$-norm
is then clear by the equivalence of both norms. The bound $\|\cdot\|_U\lesssim \Cstab{} \|\cdot\|_{V'}$
is due to inf-sup property \eqref{infsup} with inf-sup number $O(\Cstab{-1})$, and the uniform boundedness
$\|\cdot\|_{V'}\lesssim \|\cdot\|_U$ is due the uniform boundedness of bilinear form $b(\cdot,\cdot)$.
This finishes the proof of Theorem~\ref{thm}.
\qed

\section{Numerical experiments} \label{sec_numerics}
In this section, we investigate the accuracy and stability of the DPG method through numerical experiments. The field variables $\boldsymbol{u}_0=(u,w,\theta,n,q,m)$ are approximated by piecewise polynomials of degree $p$ on the mesh $\cT$, and the test functions $\boldsymbol{v} = (\delta u,\delta w, \delta\theta, \delta n, \delta q, \delta m)$ are chosen as piecewise polynomials of degree $p+\Delta p$, where $\Delta p$ is the enrichment degree. As established in Theorem~\ref{thm}, under Assumption~\ref{ass}, the DPG scheme is well-posed and yields quasi-optimal convergence, with error bounds controlled by the stability constant $C_\mathrm{stab}$. 

Since $C_\mathrm{stab}$ is defined in terms of the curvature-dependent quantities $C_n(\lambda)$, $C_q(\lambda)$, and $C_q^{(0)}(\lambda)$, the accuracy of the method for small $\epsilon$ is directly linked to their behavior. They exhibit singular growth of order $\lambda^{-2}$ as $\lambda \rightarrow 0$. Formally, this predicts error amplification in the beam limit reflecting the degeneration of the curved-arch scaling. More relevant for the intended regime of deep arches is that the constants remain nontrivial in size and may contribute to amplification even for moderate values of $\lambda$ as illustrated in Figure~\ref{fig:stability}. 

This observation motivates the use of a scaled graph norm in the test space in order to balance the effect of curvature and improve robustness. We follow \cite{gopalakrishnan_dispersive_2014} and define for $\bv \in V$ the parameter-dependent norm
\begin{equation} \label{eqn:graph norm}
\begin{split}
\| \bv \|^2_{V,\epsilon, \lambda} := &\| \delta n' + \lambda \delta q \|^2 + \| \delta q' - \lambda \delta n \|^2 + \| \delta n' + \delta q \|^2 \\ &+ \| \epsilon^2\delta n+\delta u'+ \lambda \delta w\|^2 + \|\mu \epsilon^2 \delta q+\delta w'-\lambda \delta u - \delta \theta\|^2 + \|\delta m + \delta \theta'\|^2 \\
&+ \tau_\mathrm{num} \left(\|\delta n\|^2+\|\delta q\|^2+\|\delta m\|^2+\|\delta u\|^2+\|\delta w\|^2+\|\delta \theta\|^2\right),
\end{split}
\end{equation}
where $\tau_\mathrm{num}>0$ is a numerical stabilization parameter. It should be noted, that while the enrichment degree $\Delta p=2$ is sufficient for discrete stability in the present variational setting, utilization of the graph norm involves constrained approximation and may induce parametric locking in the local problems. Therefore, a higher enrichment degree $\Delta p = 4$ is used when the graph norm is employed in the computations.

\subsection{Cantilever arch with point load at tip}
The first numerical experiment concerns a cantilever arch fixed at $x=1$ and loaded by the transversal force $f_w(x) = \delta(x)$ concentrated at $x=0$. The parameters are chosen as $\epsilon=10^{-4}$, $\mu=1$, $\lambda=6$, and the approximation degree is $p=1$. For these particular boundary conditions, the stability constant is expected to be moderate and the benefit of using the graph norm is small. Therefore, we show the results obtained with the standard test space norm for this case.

We also compare the performance of the DPG method with the state-of-the art reduced-strain finite element scheme based on the minimization of the reduced strain energy functional
\[
F(u,w,\theta) := \frac{1}{2}\left(\epsilon^2\|\Pi_h(u' + \lambda w)\|^2  
+ \mu\epsilon^2\|\Pi_h(w' - \lambda u - \theta)\|^2
   + \|\theta'\|^2\right) 
   - \vdual{f_u}{u}-\vdual{f_w}{w}
\]
over the set of continuous, piecewise linear displacement fields $(u,w,\theta)$ associated to the mesh $\cT$ and satisfying prescribed kinematic constraints. Here $\Pi_h:H^1(I) \rightarrow L_2(I)$ stands for the $L_2$-projection to piecewise constants on $\cT$ so that the method could also be interpreted as a mixed method, where the axial force $n=\epsilon^{-2}(u'+\lambda w)$ and the shear force $q=\mu^{-1}\epsilon^{-2}(w'-\lambda u - \theta)$ are approximated by piecewise constants. The strain reduction or mixed method is needed to circumvent numerical locking arising from the inextensional strain condition $u'+\lambda w=0$ and the Euler--Bernoulli condition $w'-\lambda u - \theta=0$ arising when $\epsilon$ is small. 

The approximations of the displacement and stress quantities of interest are shown side by side for the DPG method and the reduced-strain FEM in Figures \ref{fig:displacements} and \ref{fig:stresses}, respectively, for a relatively coarse mesh consisting of 8 elements.
  
\begin{figure}
\begin{tabular}{ll}
\includegraphics[scale=0.4]{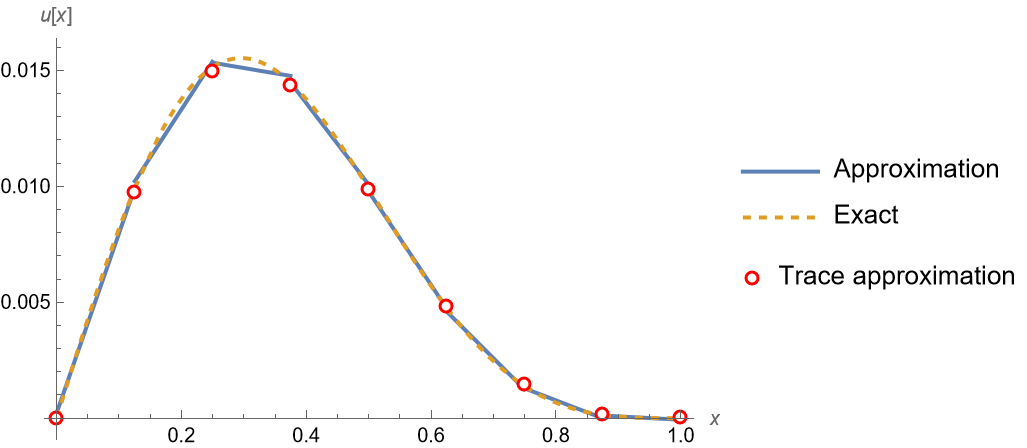}& \includegraphics[scale=0.4]{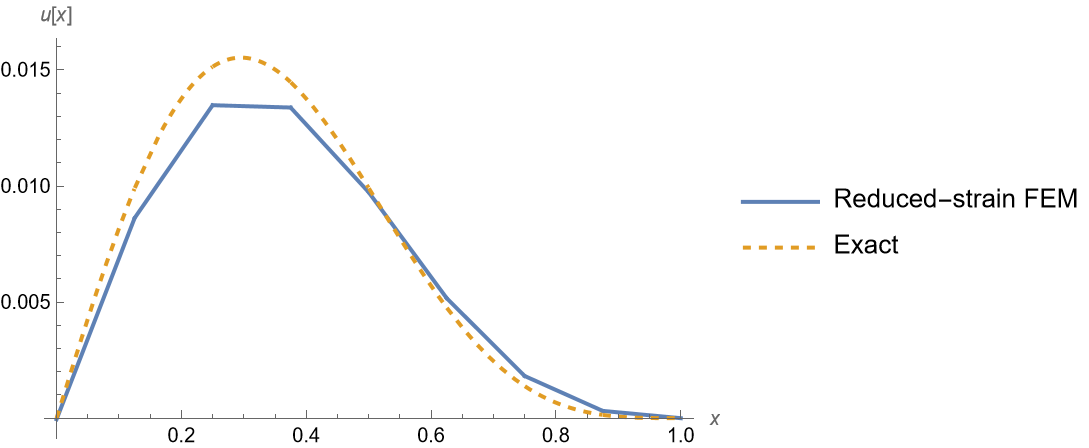} \\
\includegraphics[scale=0.4]{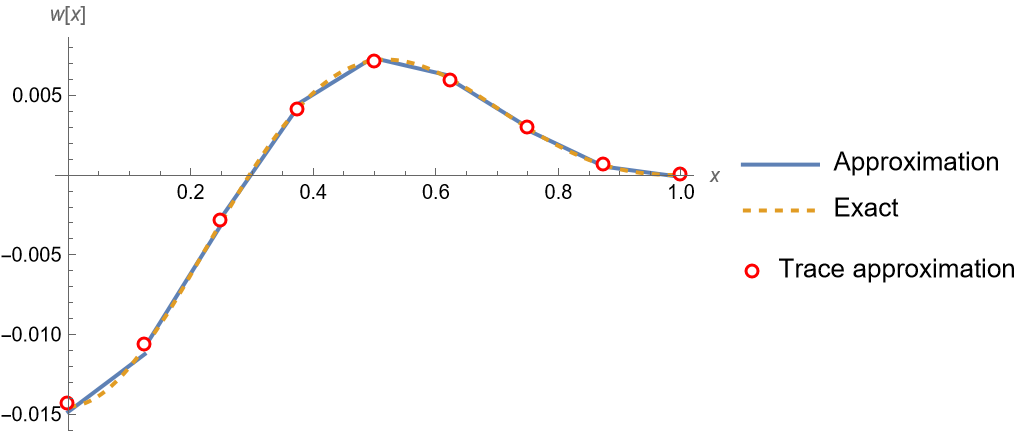}& \includegraphics[scale=0.4]{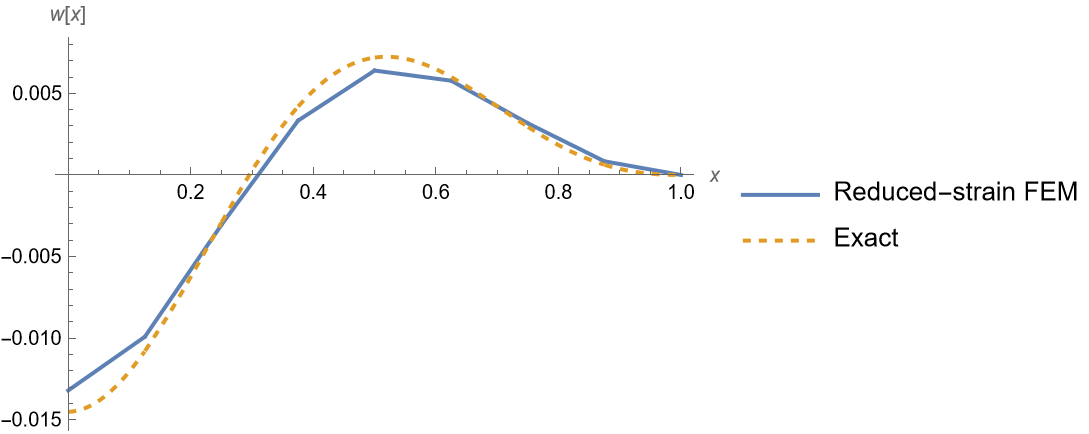} \\
\includegraphics[scale=0.4]{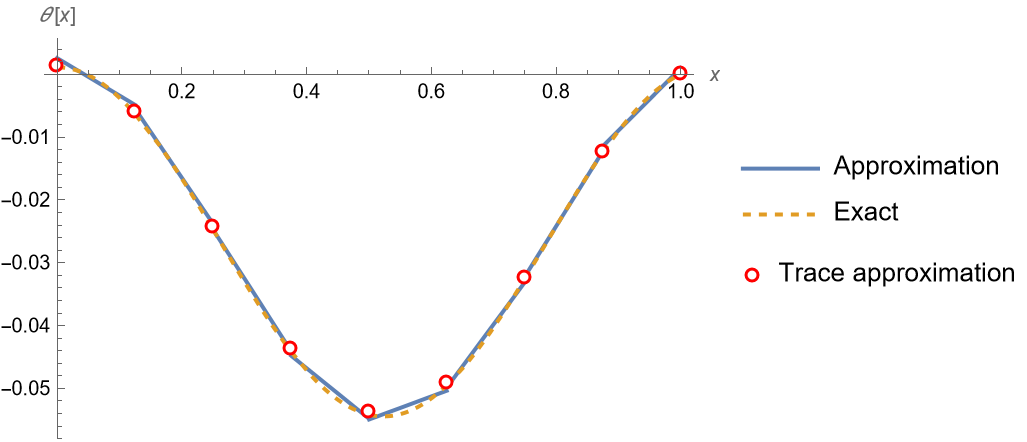}& \includegraphics[scale=0.4]{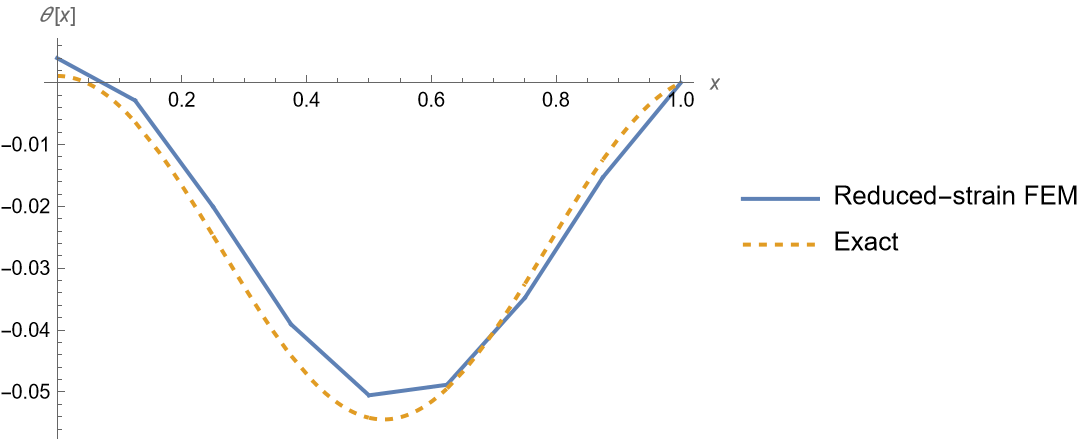}
\end{tabular}
\caption{Approximated displacement quantities of the cantilever circular arch problem with 8 elements. DPG (left) vs.~reduced-strain FEM (right).}
\label{fig:displacements}
\end{figure}

\begin{figure}
\begin{tabular}{ll}
\includegraphics[scale=0.4]{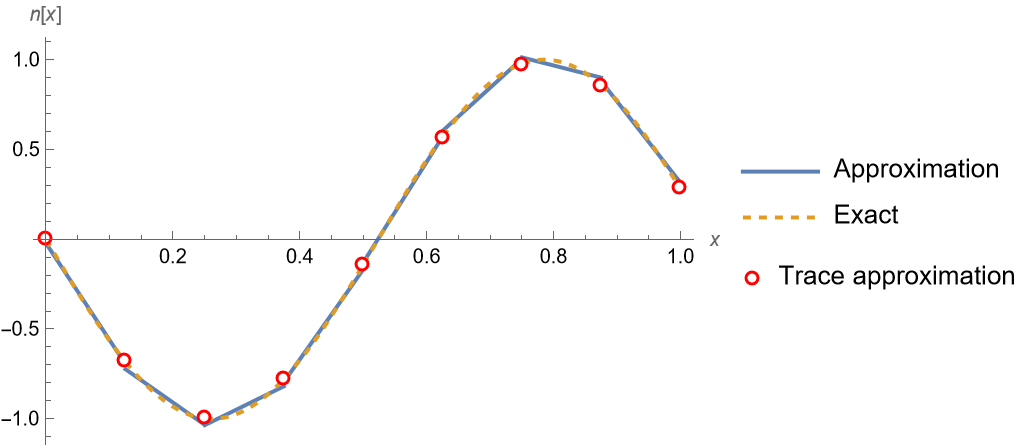}&\includegraphics[scale=0.4]{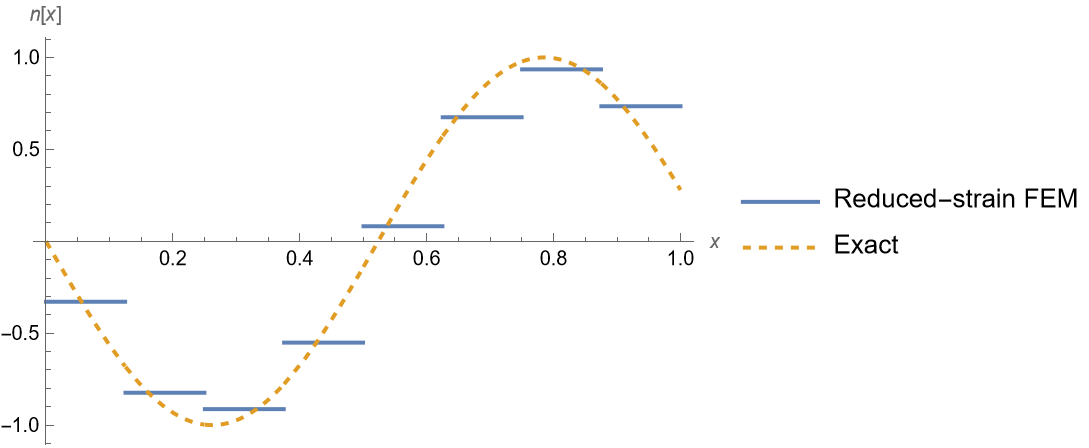} \\
\includegraphics[scale=0.4]{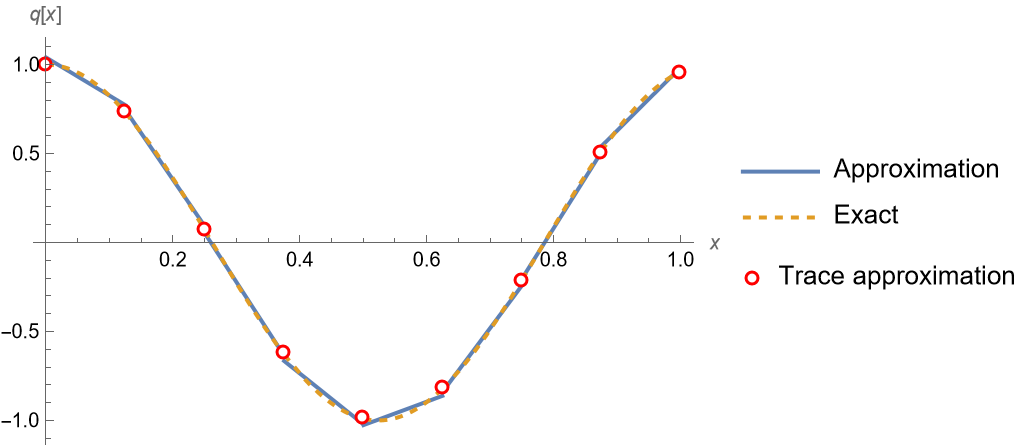}&\includegraphics[scale=0.4]{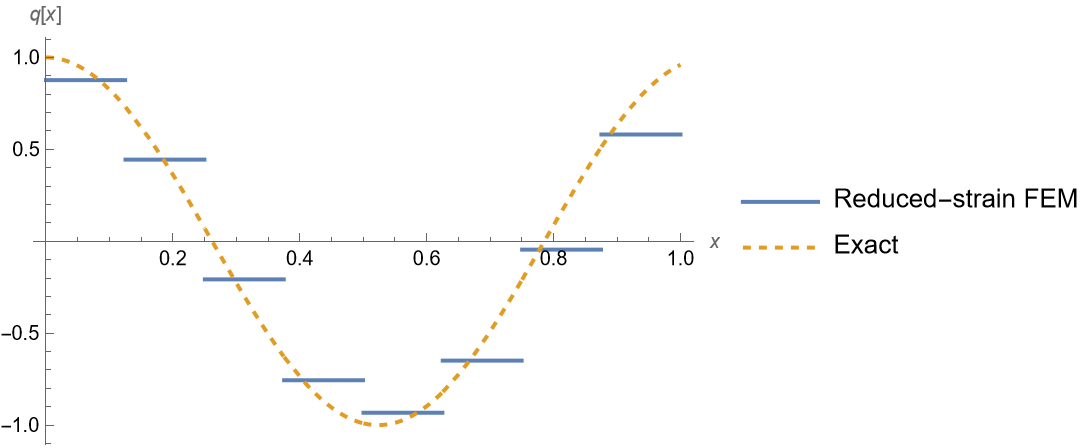} \\
\includegraphics[scale=0.4]{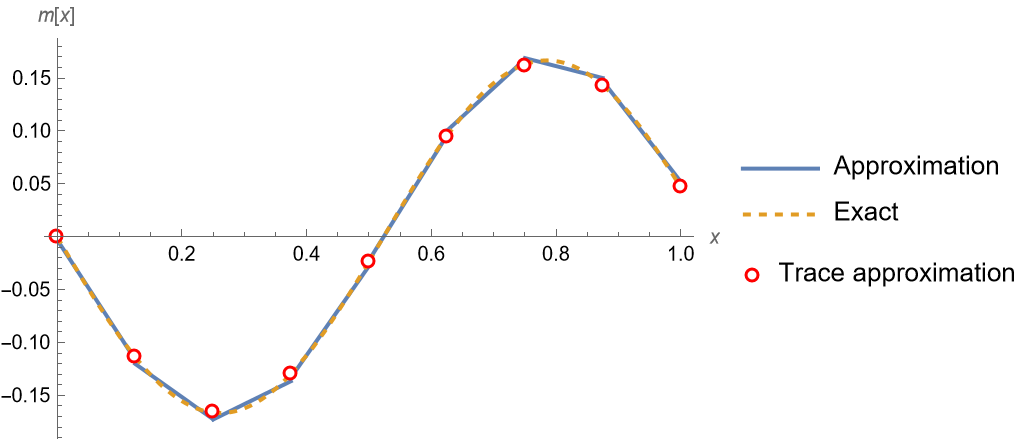}&\includegraphics[scale=0.4]{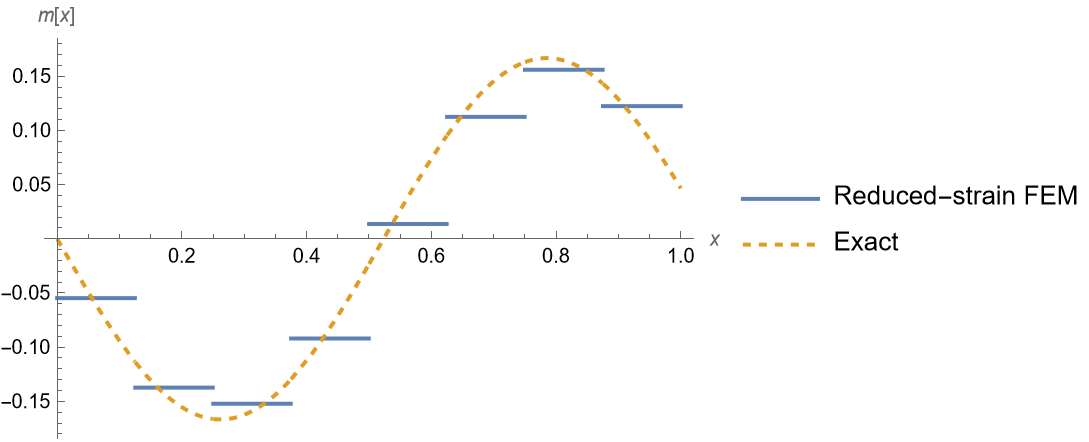}
\end{tabular}
\caption{Approximated stress quantities of the cantilever circular arch problem with 8 elements. DPG (left) vs.~reduced-strain FEM (right).}
\label{fig:stresses}
\end{figure}

It is evident that both methods deliver feasible approximations and that the ultra-weak DPG formulation represents the stresses with the same degree of accuracy as the displacements. When making judgement of the relative merits of each formulation e.g.~by comparing their accuracy, one should bear in mind that the number of global trace degrees of freedom for the ultra-weak formulation is double as compared with the number of degrees of freedom of the displacement-based method.

Figure~\ref{fig:energyconv} displays the convergence rate of the DPG method evaluated using the a posteriori error indicator $\|(T(\boldsymbol{u}-\boldsymbol{u}_h)\|_V$ built in the framework. Same parameter values are used as in the above comparison, but also the case $\mu=0$ corresponding to the Euler--Bernoulli limit with vanishing shear strains is included for comparison. In accordance with the theory, quadratic convergence rate is observed as dictated by the best-approximation error of the field variables $\boldsymbol{u}_0$. The influence of the modeling parameter $\mu$ on the accuracy is barely noticable.

\begin{figure}[ht!]
\centering
\includegraphics[scale=0.5]{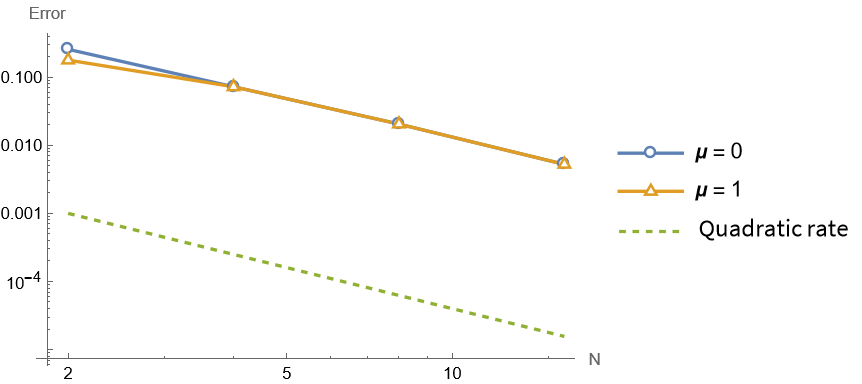}
\caption{Convergence of the full DPG solution measured by the built-in a posteriori error indicator as the number of sub-invervals $N$ increases.}
\label{fig:energyconv}
\end{figure}

\subsection{Fully clamped arch under distributed load}
As a second numerical example we consider a circular arch that is fully clamped at both ends and subjected to the distributed loads $f_u = \cos x$ and $f_w= \sin x$. The curvature parameter is set to $\lambda=3$, and we consider the Euler--Bernoulli limit ($\mu=0$). To investigate numerical accuracy, we study the relative $L_2$-errors for different quantities of interest as a function of the number of elements $N$, and we compare the effect of the slenderness parameter $\epsilon$ by considering both $\epsilon=10^{-1}$ and $\epsilon = 10^{-3}$.

This experiment is conducted using the two mentioned choices of the test space norm: the standard test space norm and the scaled graph norm \eqref{eqn:graph norm}, where the scaling parameter is fixed to $\tau_{\mathrm{num}}=10^{-5}$ based on numerical testing.  Figures~\ref{fig:L_2_conv_eps_-1} and \ref{fig:L_2_conv_eps_-3} depict the convergence behavior for these two cases. The results confirm that the scaled graph norm leads to smaller errors and the expected quadratic convergence in both cases, whereas the error level is noticeably elevated for the standard test space norm. It should be noted that, for the standard test space norm, numerical testing also revealed that \emph{a reduced enrichment degree} $\Delta p = p+1$ can partially mitigate pre-asymptotic error amplification in deep arches, but this choice may lead to ill-posedness on very coarse meshes as can be expected from discrete stability analysis. 

Overall, the experiments support the theoretical predictions, with the scaled graph norm providing robust performance across the full range of the curvature parameter.
\begin{figure}[ht!]
\begin{tabular}{ll}
\includegraphics[scale=0.5]{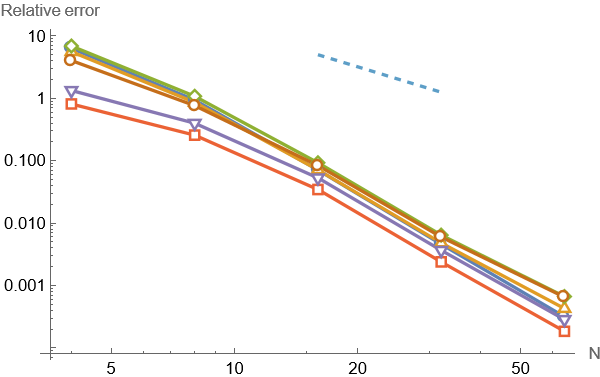}& \includegraphics[scale=0.5]{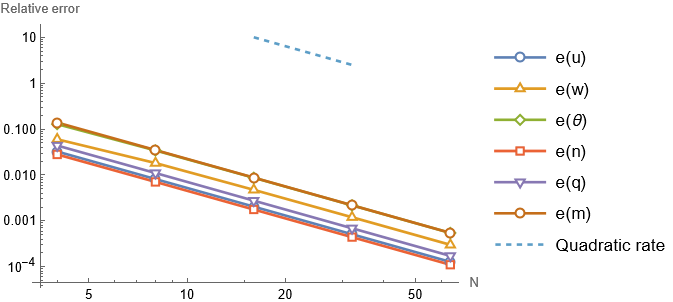} \\
\end{tabular}
\caption{Convergence of the quantities of interested of the fully clamped arch at $\epsilon=10^{-1}$. Standard test space norm (left) vs.~scaled graph norm with $\tau_{\mathrm{num}}=10^{-5}$ (right).}
\label{fig:L_2_conv_eps_-1}
\end{figure}
\begin{figure}[ht!]
\begin{tabular}{ll}
\includegraphics[scale=0.5]{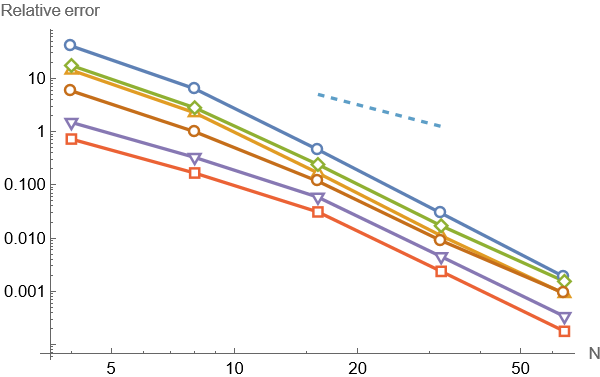}& \includegraphics[scale=0.5]{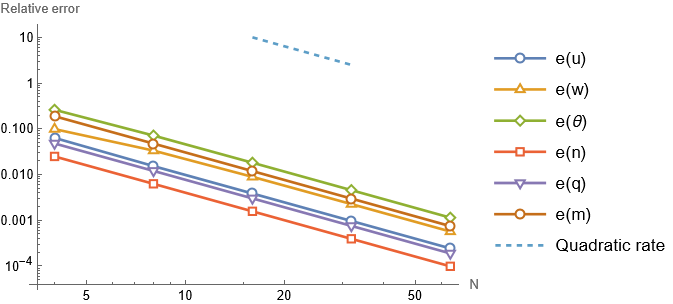} \\
\end{tabular}
\caption{Convergence of the quantities of interested of the fully clamped arch at $\epsilon=10^{-3}$. Standard test space norm (left) vs.~scaled graph norm with $\tau_{\mathrm{num}}=10^{-5}$ (right).}
\label{fig:L_2_conv_eps_-3}
\end{figure}

\newpage
\bibliographystyle{siam}
\bibliography{paper}

\end{document}